\documentclass{amsart}
\usepackage{graphicx}
\usepackage[usenames]{color}
\usepackage[normalem]{ulem}
\usepackage{cancel}
\newtheorem{thm}{Theorem}[section]

\newtheorem{lem}[thm]{Lemma}

\theoremstyle{definition}

\theoremstyle{remark}

\numberwithin{equation}{section}

\begin{document}
\title[An improvement of Zalcman's lemma in $C^n$]{An improvement of Zalcman's lemma in $C^n$}%
\author{P.V.Dovbush}%
\address{Institute of Mathematics and Computer Science of Moldova, 5 Academy  Street,  Kishinev, MD-2028,
Republic of Moldova}%
\email{peter.dovbush@gmail.com}%
\begin{abstract}
The aim of this paper is to give a proof of improving of Zalcman's lemma.
\end{abstract}
\maketitle
\section{Introduction and main result}

 A  family  $\mathcal F$  of holomorphic functions  on a  domain  $\Omega\subset C^n$
is normal in $\Omega$ if every sequence of functions  $\{f_j\} \subseteq \mathcal F$  contains either a
subsequence  which  converges  to  a limit function  $f  \ne \infty$ uniformly on each
compact subset of $\Omega,$  or a subsequence which converges uniformly to  $\infty$ on
each  compact subset.
A family $\mathcal F$ is said to be normal at a point $z_0 \in \Omega$ if it is normal in some neighborhood
of $z_0.$ A family  of analytic functions  $\mathcal F$  is normal in a domain
$\Omega$ if and only if $\mathcal F$  is  normal at each  point of $\Omega.$
For every function $\varphi$ of class $C^2(\Omega)$ we define at each point $z\in \Omega$  an hermitian form
$$ L_z(\varphi, v):=\sum_{k,l=1}^n \frac{\partial^2\varphi}{\partial z_k \partial \overline{z}_l}(z) v_k \overline{v}_l,$$
and call it the Levi form of the function $\varphi$ at $z.$
For a holomorphic function $f$ in $\Omega,$ set
\begin{equation}\label{e2}
f^\sharp (z):=\sup_{ |v|=1}\sqrt{L_z(\log(1+|f|^2), v)}
\end{equation}
This quantity is well defined since the Levi form $L_z(\log(1+|f|^2), v)$ is nonnegative for all $z\in \Omega.$
\begin{thm} (Marty's Criterion, see \cite{[DP1]}) \label{MC} A family $\mathcal F$ of functions holomorphic on $\Omega$ is normal
on $\Omega\subset C^n$ if and only if for each compact subset $K\subset \Omega$ there exists a constant $M(K)$
such that at each point $z\in K$
\begin{equation}\label{e1}
f^\sharp (z)\leq M(K)
\end{equation}
 for  all $f\in \mathcal F.$
\end{thm}
 Marty's criterion is one of the more important results in function theory  widely used for determining the
normality of a family of holomorphic functions. Marty's criterion is one of the main ingredients of
the proof of Zalcman's lemma \cite{[DP1]}. We prove the following improved result of Zalcman-Pang's \cite{[DZP]}.
\begin{thm}  \label{[HZ]} Let  $\mathcal F$ be a family of functions holomorphic on $\Omega\subset C^n.$ Then $\mathcal F$
is not normal at some point $z_0\in \Omega$ if and only if for each $\alpha \in (-1, \infty) $  there exist sequences $f_j\in \mathcal F,$ $z_j\to z_0,$ $r_j\to 0,$ such that the sequence
$$g_j(z):=r_j^\alpha f_j(z_j+r_j z)$$
 converges locally uniformly in $C^n$ to a non-constant entire function $g$ satisfying $g^\sharp(z)\leq g^\sharp(0)=1.$
 \end{thm}
In case $n=1$ this theorem was proved in Hua \cite[Lemma 6]{[HUA2]}. A similar result was proved by Chen and Gu \cite[Th.2]{[CHGU]} (see also Xue and Pung \cite{[XUPA]}, cf. Hua \cite{[HUA2]}). The special case $\alpha=0$ of Theorem \ref{[HZ]} was proved in  Zalcman
\cite[p. 814]{[ZL]} and is known as  Zalcman's rescaling lemma.  Zalcman's lemma - now upgraded to the status of theorem - was first stated at
 \cite{[ZL]}; for a state-of-the-art version, see \cite[Lemma 2]{[ZL1]}.

The plan of this paper is as follows. In Section \ref{s2}, we state and prove a number
of auxiliary results, some of which are of independent interest. In Section \ref{s3}, we
give the proof of main theorem. In Section \ref{s4}, we give two applications of  main theorem.

\section{Auxiliary results}\label{s2}
In this section, we state some known results and prove a lemma that
is required in the proofs of our results.
\begin{thm} (Hurwitz's theorem, see, e.g.  \cite[(1.5.16) Lemma, p. 24] {[NO]}) \label{[HT]} Let $\Omega$ be a domain of $C^n$ and $\{h_j\}$
a sequence of non-vanishing holomorphic
functions $h_j$ which converges uniformly on compact subsets
to a holomorphic function $h$ on $\Omega.$ Then $h$ vanishes either everywhere or nowhere.
\end{thm}

Note that  $$L_z(\log(1+|f(z)|^2), v)=\frac{|(Df(z),v)|^2}{(1+|f(z)|^2)^2}$$ on $\Omega.$
 Appealing to the Cauchy-Schwarz inequality  it is easy to show that $$(1+|f(z)|^2)f^\sharp (z)=|Df(z)|.$$

The following lemma  will play a crucial role in the proof of  Theorem \ref{[HZ]}.
\begin{lem} \label{LP} Let $f$ be a holomorphic function on the closed unit ball $\overline{B(0,1)},$
and $\alpha$ be a real number with $-1<\alpha<\infty.$ Suppose
$$ \max_{|z|\leq 1/j}{\frac{(1-j|z|)^{1+\alpha}(1+|f(z)|^2)f^\sharp(z)}{1+(1-j|z|)^{2\alpha}|f(z)|^2}}>1.$$
Then there exists a point $\xi^*,$ $|\xi^*|<1/j,$ and a real number $\rho,$ $0<\rho<1,$ such that
$$\max_{|z|\leq 1/j}{\frac{(1-j|z|)^{1+\alpha}\rho^{1+\alpha}(1+|f(z)|^2)f^\sharp(z)}{1+(1-j|z|)^{2\alpha}\rho^{2\alpha}|f(z)|^2}}=$$
 $$\frac{(1-j|\xi^*|)^{1+\alpha}\rho^{1+\alpha}(1+|f(\xi^*)|^2)f^\sharp(\xi^*)}{1+(1-j|\xi^*|)^{2\alpha}\rho^{2\alpha}|f(\xi^*)|^2}=1.$$
\end{lem}
\begin{proof} 
Set
\begin{equation}\label{equ}
\varphi(t,z):={\frac{(1-j|z|)^{1+\alpha}\rho^{1+\alpha}(1+|f(z)|^2)f^\sharp(z)}{1+(1-j|z|)^{2\alpha}\rho^{2\alpha}|f(z)|^2}}.
\end{equation}

Suppose  that $\varphi(1,z_1^*):=\max_{|z|\leq 1/j}\varphi(1,z)>1.$ Since $(1+|f(z)|^2)f^\sharp(z)$ is bounded on $|z|\leq 1/j $ we have
\begin{equation}\label{equ}
	\varphi(t,z)\leq (1-j|z|)^{1-\alpha}t^{1-\alpha} M.
\end{equation}
It follows $ \varphi(t,z)$ is continuous on $[0,1]\times\{z\in C^n : |z|\leq 1/j\}$
and $\varphi(0,z)=0$ on $\{z\in C^n : |z|\leq 1/j\}.$

Hence $\varphi(0,z_1^*)=0$ and $\varphi(1,z_1^*)>1.$
By continuity of $ \varphi(t,z)$ on $[0,1]\times\{z\in C^n : |z|\leq 1/j\},$ there exists  $\rho_1,$ $0<\rho_1<1,$ such that
$\varphi(\rho_1,z_1^*)=1.$

Repeating this procedure we
can find $\rho_m,$ $0<\rho_m<1,$ and $z_m^*,$  $ |z_m^*|<1/j,$ such that
\begin{multline}\label{equ11}
	\max_{|z|\leq 1/j}\varphi(\rho_1\ldots\rho_m,z)=\varphi(\rho_1\ldots\rho_m,z_m^*)> 1.\\
	\varphi(\rho_1\ldots\rho_m \rho_{m+1},z_m^*)=1.
\end{multline}
The sequence $\{x_m:=\rho_1\ldots\rho_m\}$ is a bounded and decreasing sequence. Then the greatest lower bound
of the set $\{x_m : m \in N\},$ say $\rho,$ is the limit of $\{x_m\}.$ The sequence  $\{z_m^*\}$ contains a subsequence, again denoted by $\{z_m^*\},$
such that $\lim_{m\to \infty}z_m^*=\xi^*.$ From (\ref{equ}) follows that $0<\rho<1$ and $|\xi^*|<1/j.$

We claim that
\begin{equation}\label{maxi}
	\max_{|z|\leq 1/j}\lim_{m\to\infty}\varphi(\rho_1\ldots\rho_m,z)=\lim_{m\to\infty}\max_{|z|\leq 1/j}\varphi(\rho_1\ldots\rho_m,z).
\end{equation}
Since $\varphi$ is continuous function on $[0,1]\times  \overline{B(0, 1/j)}$  by the Weierstrass theorem
(see \cite[Theorem (Weierstrass) p. 565]{[BEK]}) we can find $|\eta|<1/j$ and $|w_m|<1/j$ such that
   \begin{equation}\label{equ1+}
   \max_{|z|\leq 1/j}\lim_{m\to\infty}\varphi(\rho_1\ldots\rho_m,z)=\max_{|z|\leq 1/j}\varphi(\rho,z)=\varphi(\rho,\eta) ;
\end{equation}
  \begin{equation}\label{equ2}
  \varphi(\rho_1\ldots\rho_m,\eta)\leq \max_{|z|\leq 1/j}\varphi(\rho_1\ldots\rho_m,z)=\varphi(\rho_1\ldots\rho_m,w_m), \ \ m=1,2, \ldots.
\end{equation}
By the Bolzano-Weierstrass
theorem there is an infinite subsequence of $\{ w_m\},$ again denoted by $\{ w_m\},$
and $\varsigma,$ $|\varsigma|\leq
1/j ,$  such that $w_m\to \varsigma$ as $m\to\infty.$ Because $w_m\to \varsigma$ and $\rho_1\ldots\rho_m\to \rho$ as
$m\to \infty$ and  $\varphi$ is continuous function on $[0,1]\times  \overline{B(0, 1/j)}$ from (\ref{equ1+}) and (\ref{equ2}) we see
  $$\varphi(\rho,\eta)\leq \lim_{m\to\infty}\max_{|z|\leq 1/j}\varphi(\rho_1\ldots\rho_m,z)=\varphi(\rho,\varsigma)\leq$$
$$\max_{|z|\leq 1/j}\varphi(\rho,z)=\max_{|z|\leq 1/j}\lim_{m\to\infty}\varphi(\rho_1\ldots\rho_m,z)=\varphi(\rho,\eta).$$
That is, the claim (\ref{maxi}) is proved.
Combining (\ref{maxi}) and (\ref{equ1+}) we obtain
$$\max_{|z|\leq 1/j}\varphi(\rho,z)=\varphi(\rho,\xi^*)=1 \ \ (|\xi^*|<1/j).$$
The proof of the lemma is complete.
\end{proof}

\section{Proof of  main theorem}\label{s3}

\begin{proof}[Proof of Theorem \ref{[HZ]}]
  $"\Rightarrow"$ To simplify matters we assume that $z_0 =0$ and all functions under
consideration are holomorphic on the closed unit ball $\overline{B(0, 1)}.$ By Marty's criterion (Theorem \ref{MC})
$\mathcal F$ contains functions $f_j,$ $j\in N,$ satisfying
$\max_{|z|<1/(2j)} f_j^\sharp(z)>2^{1+|\alpha|}j^{3(1+|\alpha|)}.$
Since $1-j|z|>1/2$ if $|z|<1/(2j)$  there exists a $\xi_j$ with $|\xi_j|<1/j$  such that
$$
	\max_{|z|\leq 1/j}(1-j|z|)^{1+|\alpha|}f_j^\sharp(z)=(1-j|\xi_j|)^{1+|\alpha|}f_j^\sharp(\xi_j)\geq $$$$\max_{|z|\leq 1/2j}(1-j|z|)^{1+|\alpha|}f_j^\sharp(z)\geq j^{3(1+|\alpha|)}.
$$
The power function $t^{2\alpha}, t>0,$ is continuous, monotone (increasing when $\alpha>0,$ decreasing when $\alpha<0$), hence
$$(1-j|z|)^{2\alpha}(1+|f(z)|^2)\geq 1+(1-j|z|)^{2\alpha}|f(z)|^2 \ \ (-1<\alpha\leq 0 \ \  \textrm{ arbitrary})$$
and
$$1+(1-j|z|)^{2\alpha}|f(z)|^2 \leq  [1+|f(z)|^2]\ \ (0<\alpha<\infty \ \ \textrm{ arbitrary})$$
 we have
\begin{equation} \label{pa11} \frac{(1-j|\xi_j|)^{1+\alpha}(1+|f_j(\xi_j)|^2)f_j^\sharp(\xi_j)}{1+(1-j|\xi_j|)^{2\alpha}|f_j(\xi_j)|^2}
> (1-j|\xi_j|)^{1+|\alpha|}f_j^\sharp(\xi_j)> j^{3(1+|\alpha|)}.
\end{equation}
Hence
$$ \max_{|z|\leq 1/j}{\frac{(1-j|z|)^{1+\alpha}(1+|f_j(z)|^2)f_j^\sharp(z)}{1+(1-j|z|)^{2\alpha}|f_j(z)|^2}}
>1.$$
According to Lemma \ref{LP}, there exists $\xi_j^*,$ $|\xi_j^*|<1/j,$ and $\rho_j,$ $0<\rho_j<1,$ such that
$$
	\max_{|z|\leq 1/j}{\frac{(1-j|z|)^{1+\alpha}\rho_j^{1+\alpha}(1+|f_j(z)|^2)f_j^\sharp(z)}{1+(1-j|z|)^{2\alpha}\rho_j^{2\alpha}|f_j(z)|^2}}=$$ $$\frac{(1-j|\xi_j^*|)^{1+\alpha}\rho_j^{1+\alpha}(1+|f_j(\xi_j^*)|^2)f_j^\sharp(\xi_j^*)}{1+(1-j|\xi_j^*|)^{2\alpha}\rho_j^{2\alpha}|f_j(\xi_j^*)|^2}=1.
$$
Therefore inequality (\ref{pa11}) shows that
$$
1=\frac{(1-j|\xi_j^*|)^{1+\alpha}\rho_j^{1+\alpha}(1+|f_j(\xi_j^*)|^2)f_j^\sharp(\xi_j^*)}{1+(1-j|\xi_j^*|)^{2\alpha}\rho_j^{2\alpha}|f_j(\xi_j^*)|^2}\geq $$
$$ \frac{(1-j|\xi_j|)^{1+\alpha}\rho_j^{1+\alpha}(1+|f_j(\xi_j)|^2)f_j^\sharp(\xi_j)}{1+(1-j|\xi_j|)^{2\alpha}\rho_j^{2\alpha}|f_j(\xi_j)|^2}\geq$$
$$
 (1-j|\xi_j|)^{1+|\alpha|}\rho_j^{1+|\alpha|}f_j^\sharp(\xi_j)\geq\rho_j^{1+|\alpha|}j^{3(1+|\alpha|)} \ \ (|\xi_j|<1/j).	
$$
It follows
\begin{equation}\label{pa111}
	\Big(\frac{1}{j}\Big)^{3}\geq\rho_j\to 0.
\end{equation}
Put
$$r_j=(1-j|\xi_j^*|)\rho_j\to 0.$$
Set $$h_j(z)= r_j^{\alpha}f_j(\xi_j^*+ r_jz).$$
We claim that appropriately chosen
subsequences $z_k =\xi_{j_k},$ $\rho_k=r_{j_k},$ and $g_k =h_{j_k}$ will do.
First of all, $h_j(z)$ is defined on $|z|<\frac{1}{j\rho_j},$ hence on $|z|<j,$ since
$$ |\xi_j^*+r_jz|\leq |\xi_j^*|+r_j|z|\leq|\xi_j^*|+r_j\frac{1-j|\xi_j^*|}{jr_j}=\frac{1}{j}.$$
By the invariance of the Levi form under  biholomorphic mappings, we have
$$L_{z}(\log(1+|h_j|^2),v)=L_{\xi_j^*+r_jz}(\log(1+|h_j|^2),r_jv)$$
and hence $$h_j^\sharp (z)=r_jh_j^\sharp (\xi_j^*+r_jz).$$
Since $r_j=(1-j|\xi_j^*|)\rho_j$ a simple
computations shows that
 $$h_j^\sharp (z)=\frac{r_jr_j^{\alpha}(1+|f_j(\xi_j^*+r_jz)|^2)f_j^\sharp (\xi_j^*+r_jz)}{1+r_j^{2\alpha}|f_j(\xi_j^*+r_jz)|^2}=
$$
$$\frac{(1-j|\xi_j^*|)^{1+\alpha}\rho_j^{1+\alpha}(1+|f_j(\xi_j^*+r_jz)|^2)f_j^\sharp (\xi_j^*+r_jz)}{1+[(1-j|\xi_j^*|)/(1-j|\xi_j^*+r_jz|)]^{2\alpha}(1-j|\xi_j^*+r_jz|)^{2\alpha}\rho_j^{2\alpha}|f_j(\xi_j^*+r_jz)|^2}=$$
$$\frac{(1-j|\xi_j^*|)^{1+\alpha}}{(1-j|\xi_j^*+r_jz|)^{1+\alpha}}\cdot \frac{(1-j|\xi_j^*+r_jz|)^{1+\alpha}\rho_j^{1+\alpha}(1+|f_j(\xi_j^*+r_jz)|^2)f_j^\sharp (\xi_j^*+r_jz)}{\Big[1+\Big(\frac{1-j|\xi_j^*|}{1-j|\xi_j^*+r_jz|}\Big)^{2\alpha}\cdot(1-j|\xi_j^*+r_jz|)^{2\alpha}\rho_j^{2\alpha}|f_j(\xi_j^*+r_jz)|^2\Big]}.$$
Bearing in mind Lemma \ref{LP} it is easy to see that $h_j^\sharp (0)=1.$ Since

 $$\frac{1}{1+1/j}\leq\frac{1-j|\xi_j^*|}{1-j|\xi_j^*+r_jz|}\leq \frac{1}{1-1/j} $$
we have

$$
1+\Big(\frac{1-j|\xi_j^*|}{1-j|\xi_j^*+r_jz|}\Big)^{2\alpha}\cdot(1-j|\xi_j^*+r_jz|)^{2\alpha}\rho_j^{2\alpha}|f_j(\xi_j^*+r_jz)|^2 \geq$$
$$\Big(\frac{1}{1-1/j}\Big)^{2\alpha}\cdot\Big[1+(1-j|\xi_j^*+r_jz|)^{2\alpha}\rho_j^{2\alpha}|f_j(\xi_j^*+r_jz)|^2\Big] \ \ (-1<\alpha\leq 0 \ \  \textrm{ arbitrary}) $$
and

$$
1+\Big(\frac{1-j|\xi_j^*|}{1-j|\xi_j^*+r_jz|}\Big)^{2\alpha}\cdot(1-j|\xi_j^*+r_jz|)^{2\alpha}\rho_j^{2\alpha}|f_j(\xi_j^*+r_jz)|^2 \geq$$
$$\Big(\frac{1}{1+1/j}\Big)^{2\alpha}\cdot\Big[1+(1-j|\xi_j^*+r_jz|)^{2\alpha}\rho_j^{2\alpha}|f_j(\xi_j^*+r_jz)|^2\Big] \ \ (0<\alpha< \infty \ \  \textrm{ arbitrary}). $$

 From the above inequalities and Lemma \ref{LP} we infer that\footnote{) sgn denotes the signum function (i.e., $sgn(0)=0,$ $sgn(\alpha)=1$  if $\alpha>0$  and $-1$ if $\alpha<0$).})
$$h_j^\sharp (z)$$
$$\leq\Big(1+\frac{sgn(\alpha)}{j}\Big)^{2\alpha}\cdot\Big(\frac{1-|\xi_j^*|}{1-j|\xi_j^*+r_jz|}\Big)^{1+\alpha}\cdot  \frac{(1-j|\xi_j^*+r_jz|)^{1+\alpha}\rho_j^{1+\alpha}(1+|f_j(\xi_j^*+r_jz)|^2)f_j^\sharp(\xi_j^*+r_jz)}{1+(1-j|\xi_j^*+r_jz|)^{2\alpha}\rho_j^{2\alpha}|f_j(\xi_j^*+r_jz)|^2}
$$
 $$ =\Big(1+\frac{sgn(\alpha)}{j}\Big)^{2\alpha}\cdot\Big(\frac{1-j|\xi_j^*|}{1-j|\xi_j^*+r_jz|}\Big)^{1+\alpha}\cdot 1\leq \Big(1+\frac{sgn(\alpha)}{j}\Big)^{2\alpha}\cdot\Big(\frac{1}{1-1/j}\Big)^{1+\alpha}
 $$ for all $|z|<j.$
For every $m \in N$ the sequence $\{h_j\}_{j>m}$ is normal
in $B(0, m)$ by Marty's criterion (Theorem \ref{MC}). The well-known Cantor diagonal process yields a subsequence
$\{g_k=h_{j_k}\}$ which converges uniformly on every ball $B(0, R).$ The limit function
$g$ satisfies $g^\sharp(z)\leq \limsup_{j\to\infty} h_j^\sharp(z) \leq 1=g^\sharp(0).$ Clearly, $g$ is non-constant because
$g^\sharp(0) \neq 0.$

$"\Leftarrow"$ Take $\alpha=0.$
Suppose that there exist sequences $f_j\in \mathcal F,$ $z_j\to 0,$ $\rho_j\to 0,$
 such that the  sequence $$g_j(z)=f_j(z_j+\rho_j z) $$
 converges locally uniformly in $C^n$ to a non-constant entire function $g$ satisfying $g^\sharp(z)\leq g^\sharp(0)=1,$  but $\mathcal F$
is normal.
By Marty's criterion (Theorem \ref{MC}) there exists a constant $M > 0$ such that
$$\max_{|z|\leq 1/2} f_j^\sharp(z)<M$$ for all $j.$
Since $z_j\to  0,$ $\rho_j\to 0,$ then for $|z|<1/2$ and $j$  sufficiently large, we have
$$|z_j+\rho_j z| \leq |z_j|+\rho_j |z|\leq |z_j|+\rho_j/2 <1/2.$$ Thus
$$g_j^\sharp(z)=f_j^\sharp(z_j+\rho_jz)\rho_j\leq M\rho_j\to 0 \ \ (|z|<1/2).$$
This implies that $g^\sharp(0)=0,$ which is a contradiction to $g^\sharp(0)=1.$
\end{proof}
\section{Applications}\label{s4}
Let us illustrate the use of improved Zalcman's lemma  by showing first how it can be used to derived the following theorem.
\begin{thm} \label{suf} Let some $\varepsilon > 0$ be given and set
$$\mathcal F=\{f \textrm{ holomorphic in }\Omega : f^\sharp(z)> \varepsilon \textrm{ for all } z \in \Omega \}.$$
Then $\mathcal F$ is normal in $\Omega.$
\end{thm}
\begin{proof} To obtain a contradiction suppose that  $\mathcal F$  is not normal in point $z_0\in \Omega.$ Without restriction we may assume $z_0=0.$
If $\mathcal F$ is not normal at $0,$ it follows from Theorem \ref{[HZ]}  that there exist $f_j\in \mathcal F,$ $z_j\to 0,$ $\rho_j\to 0,$
 such that the  sequence $$g_j(z)=\rho_j^2\cdot f_j(z_j+\rho_j z) $$
 converges locally uniformly in $C^n$ to a non-constant entire function $g$ satisfying $g^\sharp(z)\leq g^\sharp(0)=1.$
  Since $g$ the non-constant entire function  it follows that exists $a\in \Omega$ such that $|g(a)|> 0.$ Hence $|g_j(a)|\neq 0$ for all $j$ sufficiency large
 $$ 1\geq g_j^\sharp(a)=\frac{\max_{|v|=1}|(Df_j(z_j+\rho_j a),v)|\cdot|g_j(a)|^2}{\rho_j\cdot|f_j(z_j+\rho_j a)|^2\cdot(1+|g_j(a)|^2)}\geq$$
 $$\frac{f_j^\sharp(z_j+\rho_j a)}{\rho_j}\cdot\frac{|g_j(a)|^2}{1+|g_j(a)|^2}\geq\frac{\varepsilon}{\rho_j}\cdot\frac{|g_j(a)|^2}{1+|g_j(a)|^2}.$$
 The right-hand side of this inequality tends to infinity as $j\to \infty,$ a contradiction. This contradiction shows that $\mathcal F$ is normal in $\Omega.$
\end{proof}
For families of holomorphic functions which do not vanish, we have the following theorem.
\begin{thm} \label{ZLP01} Let $\mathcal F$ be a family of zero-free holomorphic functions in a domain $\Omega \subset C^n.$
The statement of Theorem \ref{[HZ]} remains valid if   $-1\leq\alpha<\infty$ is replaced with $-\infty<\alpha< \infty.$
\end{thm}
\begin{proof}[Proof of Theorem \ref{ZLP01}]

It need to consider only the case $-\infty< \alpha <0.$
Since  a family $\{1/f, f\in \mathcal F\}$ conforms to the hypotheses of Theorem \ref{ZLP01} the earlier argument shows   that
 there exist sequences $1/f_j,$ $z_j\to z_0,$ $r_j\to 0,$
 such that the sequence $$g_j(z):=\frac{r_j^\alpha}{ f_j(z_j+r_j z) }\ \ (0\leq\alpha<\infty \textrm \ \ arbitrary)$$
 converges locally uniformly in $C^n$ to a non-constant entire function $g$ satisfying $g^\sharp(z)\leq g^\sharp(0)=1.$
By  Hurwitz's theorem  either $g\equiv 0$ or $g$
never vanishes. Since $ g^\sharp(0)=1$  it is easy to see that  $g$ never vanishes then $1/g$ is entire function in $C^n.$
It follows $r_j^{-\alpha} f_j\to 1/g$ uniformly in $C^n.$
Since Levi form vanishes for any
pluriharmonic function,
$$L_z(\log(1+|1/g|^2), v)= L_z(\log(1+|g|^2), v)-2L_z(\log|g|, v)=L_z(\log(1+|g|^2), v).$$
Therefore, $$g^\sharp (z)=(1/g)^\sharp(z).$$
For every $z\in C^n$ we have $g^\sharp(z)\leq g^\sharp(0)=1,$ hence
$$(1/g)^\sharp(z)\leq (1/g)^\sharp(0)=1.$$
The case $-\infty\leq \alpha <0$ is proved.
This completes the proof of the theorem.
\end{proof}

\end{document}